\documentclass[10pt]{article}
\usepackage{graphicx}
\usepackage{seqsplit}
\usepackage{latexsym}
\usepackage{amsmath}
\usepackage{amssymb}
\usepackage{amsthm}
\usepackage[pdftex]{hyperref}
 \usepackage{helvet}

\DeclareMathOperator{\tdeg}{tdeg}

\def \lmk {l_{\mathbf{m},k}}
\def \Q {{\mathbb Q}}

\def \W {{\mathbb W}}
\def \Z {{\mathbb Z}}

\def \0 {{\mathbf 0}}

\newtheorem{theorem}{Theorem}[section]
\newtheorem{cor}[theorem]{Corollary}
\newtheorem{lemma}{Lemma}[section]
\newtheorem{pro}[theorem]{Proposition}

\newtheorem{definition}[theorem]{Definition}

\newtheorem{ex}[theorem]{Example}

\date{}

\title{Fixed Divisor of a Multivariate Polynomial and Generalized Factorials in 
Several Variables}
\author {Devendra Prasad, Krishnan Rajkumar, A. Satyanarayana Reddy\\
 dp742@snu.edu.in, 	krishnan@mail.jnu.ac.in,\\ satyanarayana.reddy@snu.edu.in.
  }
  
\begin{document}
\maketitle
 

\begin{abstract}
We define new generalized factorials in several variables over an arbitrary subset $\underline{S} \subseteq R^n,$ where $R$ is a Dedekind domain and $n$ is a positive integer. We then study the properties
of the fixed divisor $d(\underline{S},f)$ of a multivariate polynomial $f \in R[x_1,x_2, \ldots, x_n]$.
We generalize the results of Polya, Bhargava, Gunji \& McQuillan and strengthen that of Evrard, all of which relate the fixed divisor to generalized factorials of $\underline{S}$. We also express $d(\underline{S},f)$ in terms of the images $f(\underline{a})$ of finitely many elements $\underline{a} \in R^n$, generalizing a result of Hensel, and in terms of the coefficients of $f$ under explicit bases.
\end{abstract}
\textbf{keywords} {Fixed divisor, Generalized factorials, Dedekind domain}
\maketitle

\section{Introduction}\label{s1}

 Let $R$ be a Dedekind domain, $n$ a positive integer and 
$\underline{S}\subseteq R^n$ be an arbitrary subset. Let $f \in R[x_1,x_2, 
\ldots, x_n]=R[\underline{x}]$ be a  
  polynomial in $n$ variables. The fixed divisor of $f$ over $\underline{S}$, 
denoted 
$d(\underline{S},f)$, is defined as the ideal in $R$ generated 
by the values of   $f$ on $\underline{S}.$ 

\medskip

The study of $d(\underline{S},f)$ appears to have been initiated by Hensel 
\cite{Hensel} (see  \cite{Dickson} also) in 1896  where he proved the 
following
\begin{theorem}\label{Hensel}
Let $f \in \mathbb{Z}[\underline{x}]$ be a polynomial with degree $m_i$ in $x_i$ 
for $i=1,2,\ldots , n$. Then $d(\Z^n,f)$ equals the g.c.d. of the values 
$f(r_1,r_2, \ldots , r_n)$ where each $r_i$ ranges over $m_i+1$ consecutive 
integers.
\end{theorem}

In case when $R$ is  a Dedekind domain with finite norm property, Polya \cite{polya} (see \cite{Narkiewicz} also)
explicitly constructed a sequence of ideals $A_k$ for each integer $k$, which 
served as a bound for the fixed divisor of a univariate polynomial. He proved
\begin{theorem}\label{Pol}
  Let $R$ be  a Dedekind domain with finite norm property, $I \subseteq R$ be a 
proper ideal and $k \geq 2$. Then  $I$ is the fixed divisor of some primitive 
polynomial of degree $k$ in $R[x]$ over $R$   iff $I$ divides $A_k$.
\end{theorem}

When $\underline{S}=R=\Z$, $A_k=k!$ is the upper bound. Later Cahen \cite{Cahen 
fd} relaxed the condition of finite norm property in the above theorem.

\medskip

Gunji and McQuilllan \cite{Gunji}, \cite{Gunji1} extended Theorem \ref{Pol} in two different aspects. For the one variable case and $R$ any number 
field, they generalized  the result of Polya when $\underline{S}$ is the coset 
of any ideal. In the multivariate case for $R=\Z$, they considered the case when 
$\underline{S}$ is the Cartesian product of arithmetical progressions and proved 
the following

\begin{theorem} \label{Gun} Let $A_i=\lbrace 
sa_i+b_i\rbrace_{s \in \mathbb{Z}}$ be an 
arithmetic progression with $a_i, b_i \in \mathbb{Z}$  for $i=1,2,\ldots n$. Let 
$\underline{A}= A_1 \times A_2 \times \cdots \times A_n$ and $d\in \Z$ be any 
integer.
Then there exists a primitive 
polynomial $f \in \Z[\underline{x}]$ in $n$ variables with degree $m_i$ in each 
variable $x_i$ such that $d(\underline{A},f) = (d)  $ iff $d$ 
divides $\prod_{i=1}^n m_i! a_i^{m_i}$.

\end{theorem}

\medskip

In the case of one variable, the complete generalization of Theorem \ref{Pol} 
for a general subset $S \subseteq R$ was given by Bhargava \cite{Bhar2} (also see 
\cite{Bhar3} and \cite{Bhar1}) by introducing the notion of generalized 
factorial $k!_S$ to replace $A_k$ of Theorem \ref{Pol}. For an 
excellent exposition of the history and various definitions of $k!_S$  see Cahen 
and Chabert \cite{Cahen factorial} (also see \cite{Wood}, \cite{Bhar1}, 
\cite{Bhar3}). Bhargava also obtained a formula for $d(S,f)$ in terms of the 
coefficients of $f$.

\medskip

The case of a multivariate polynomial for general $\underline{S}$ was addressed 
by Evrard \cite{Ev} by generalizing Bhargava's factorial in several variables. 
In order to define this factorial, we need the notion of the ring of integer 
valued polynomials on $\underline{S}$.

Let $K$ be the field of fractions of $R$. Then for $\underline{S} \subseteq 
R^n$, the ring of integer valued polynomials on $\underline{S}$ is defined as 
$$\mathrm{Int}(\underline{S},R)= \lbrace f \in K[\underline{x}] 
:f(\underline{S}) 
\subseteq R \rbrace. $$ 

These rings have been extensively studied in the last few decades. Cahen and 
Chabert \cite{Cahen} is a good reference for this. We also need the following 
notation
$$\mathrm{Int}_k(\underline{S},R)=\{f\in \mathrm{Int}(\underline{S},R) : 
\mathrm{\ total\ degree\ of\ } f\le k\}.$$ 
 
The generalized multivariable factorial is defined as follows.
\begin{definition} For each $k \in \mathbb{N}$ and $\underline{S} \subseteq 
R^n$, the generalized factorial of index 
$k$ 
is defined 
 by $$k!_{\underline{S}}=\lbrace a \in R : a \:  
\mathrm{Int}_k(\underline{S},R)\subseteq 
R[\underline{x}]\rbrace.$$
\end{definition}

There are several properties of $k!_{\underline{S}}$ that make it a good 
generalization of Bhargava's factorial to several variables. These properties will 
be discussed in Section \ref{s2}. Using this factorial Evrard \cite{Ev} proved 
the following generalization of Theorem \ref{Pol}.

\begin{theorem}\label{Evrard's bound}
Let $f$ be a primitive polynomial of total degree $k$ in $n$ variables  and 
$\underline{S} \subseteq R^n$, then $d(\underline{S},f)$ divides 
$k!_{\underline{S}}$ and this is sharp.
\end{theorem}

The sharpness of the statement denotes (and will denote in next sections) the 
existence of a  polynomial $f$ 
satisfying the conditions of the theorem
such that  $d(\underline{S},f)=k!_{\underline{S}}.$ This sharpness was obtained 
by using the notion of $\nu$-ordering in $\underline{S}$ (originally due to 
Bhargava \cite{Bhar3}). This notion also proved to be very useful in the computation 
of $k!_{\underline{S}}$ (Proposition 26, \cite{Ev}) and in testing for membership 
of a polynomial in $\mathrm{Int}_k(\underline{S},R)$(Corollary 17, \cite{Ev}).

\medskip

 Observe that Theorems \ref{Hensel} and \ref{Gun} consider the notion 
of partial degrees for a multivariate polynomial as compared to Theorem \ref{Evrard's bound} which considers the notion of total degree. In this paper we take into account \textit{both} these notions of degree and obtain a new generalization of Bhargava's factorial 
in several variables. This factorial denoted 
$\Gamma_{\mathbf{m},k}(\underline{S})$ is indexed by two parameters $\mathbf{m} 
\in \W^n$, $k \in \W$ and reduces to $k!_{\underline{S}}$ in special cases. Here 
(and throughout this paper) $\W$ denotes the set of non-negative integers.

We use this factorial to get a generalization of Polya's result (Theorem 
\ref{Pol}) which is sharper than Theorem \ref{Evrard's bound}. For $\underline{S}$ not contained in an algebraic subset of $K^n$, we define an 
analogue of $\nu$-ordering which is helpful in the computation of 
$\Gamma_{\mathbf{m},k}(\underline{S})$ and also improves the criteria for 
membership in $\mathrm{Int}(\underline{S},R)$ in some cases. 

We also obtain a generalization of Hensel's result (Theorem \ref{Hensel}) for such $\underline{S}$, in 
which we show that $d(\underline{S},f)$ is the g.c.d. of finitely many values 
of $f$ (at explicitly constructed elements). To our knowledge, excepting the 
case of a univariate polynomial over a discrete valuation ring (DVR) in 
\cite{Bhar2}, this is the first time that this question has been addressed in 
the general case. 
Finally we show that at most two values of $f$ are sufficient to determine 
$d(\underline{S},f)$!

\medskip

The organization of this paper is as follows. In section \ref{s2} we will define 
$\Gamma_{\mathbf{m},k}(\underline{S})$, establish many of its properties and 
also discuss the advantages of these factorials  in this section and the next. In section \ref{s3} we consider 
the case when $\underline{S}$ is a Cartesian product and obtain a formula for 
$d(\underline{S},f)$ in terms of coefficients of $f$. We shall also compute and 
compare the various factorials in this case.  In section 
\ref{s4} we determine $d(\underline{S},f)$ by finitely many values of $f$.

\section{New generalized factorials in several variables}\label{s2}

 We start this section by recalling some of the properties of the factorial 
function which play an important role in various applications. See Chabert 
\cite{Chabert} for more information regarding these applications as well as 
generalizations of these properties in other contexts.

\medskip

\noindent
$\textbf{Property A.}$ For all $k, l \in \mathbb{N}$,
$k!l!$ divides $(k+l)!$.

\noindent
$\textbf{Property B.}$ For every sequence $x_0, x_1, \ldots, x_n$ of $n+1$ 
integers, the product
$\prod_{0 \le i<j \le n}{(x_j -x_i) }$ is divisible by $1!2!\ldots n!$. 

\noindent
$\textbf{Property C.}$ For every primitive polynomial $f \in \Z[x]$ of degree 
$n$,
$d(\Z,f )$ divides $n! $.

\noindent
$\textbf{Property D.}$ For every integer-valued polynomial $g \in \Q[x]$ of degree $ n$,
$n!g \in \Z[x] . $

\medskip

Note that property C is Polya's result (Theorem \ref{Pol}) in the case $R=\Z$. 
As mentioned in Section \ref{s1}, $k!_{\underline{S}}$ satisfies generalizations 
of each of these properties. We introduce a new generalized factorial which also 
satisfies all of the above properties in the general setting and coincides with 
$k!_{\underline{S}}$ in special cases.

The new generalized factorial can be defined by starting from any of Properties 
B, C or D. For instance, the starting point for Bhargava \cite{Bhar3} is a 
generalization of Property B while Evrard \cite{Ev} starts with that of Property 
D. Each of these definitions has its advantages and all turn out to be 
equivalent. In this section we will follow the latter approach as the exposition 
becomes very concise. We note that almost all of the results and proofs in 
\cite{Ev} carry over to this setting (with appropriate restrictions on the 
degree of the polynomials involved). However, for the sake of completeness we 
include the alternative proofs.

\medskip
 
Let us fix the notation for the rest of the paper. Let $\mathbf{i}\in 
\mathbb{W}^n$ denote the $n$-tuple $(i_1,i_2,\ldots,i_n)$, with 
$\mathbf{0}=(0,0,\ldots,0)$. 
Let $\mathbf{i} \leq \mathbf{j}$ denote the condition that $i_k \leq j_k$ for 
each component $k=1,2,\ldots,n$. The degree of the multivariate polynomial $f$, denoted by $\deg(f)$, is defined as the $n$-tuple $\mathbf{m}=(m_1,m_2,\ldots,m_n)$ where $m_i$ is the partial
degree of $f$ in $x_i$. Note that this definition is different from the total 
 degree of the polynomial, denoted by $\tdeg(f)$ for the remainder of this 
paper. We  call $f$ of type $(\mathbf{m},k)$ if $\deg(f)=\mathbf{m}$ and $\tdeg(f)=k$.

\medskip

For any $\mathbf{m} \in \mathbb{W}^n$ and   $k \in \mathbb{W}$, define  
$$\mathrm{Int}_{\mathbf{m},k}(\underline{S},R)=\{f\in 
\mathrm{Int}(\underline{S},R) : 
\deg(f) \leq \mathbf{m},\ \tdeg(f) \le k \}.$$

  Similarly, we can also define $\mathrm{Int}_{\mathbf{m}}(\underline{S},R)$ as 
above without any condition on total degree. Thus we always have 
$\mathrm{Int}_{\mathbf{m},k}(\underline{S},R) = \mathrm{Int}_k(\underline{S},R) 
\cap\mathrm{Int}_{\mathbf{m}}(\underline{S},R)$.

\begin{definition} For $\mathbf{m} \in \mathbb{W}^n, k \in \mathbb{W},$  and 
$\underline{S} \subseteq 
R^n$, the generalized factorial of index 
$k$ with respect to $\mathbf{m}$ is defined 
 by $$\Gamma_{\mathbf{m}, k}(\underline{S}) =\lbrace a \in R : a \:  
\mathrm{Int}_{\mathbf{m},k}(\underline{S},R)  \subseteq 
R[\underline{x}]\rbrace.$$
\end{definition}

It follows that $\Gamma_{\mathbf{m}, k}(\underline{S})$ always divides 
$k!_{\underline{S}}$. For fixed $k$ and varying ${\mathbf{m}}$,  
the factorials $\Gamma_{\mathbf{m}, k}(\underline{S})$ will be identical 
to $k!_{\underline{S}}$ when each $m_i \geq c$, for some constant $c=c(k,\underline{S})$ (for example, $c=k$ will do). A similar phenomenon occurs if we fix ${\mathbf{m}}$ 
and vary $k$.

We now arrive at the central result of this section which is a strengthening of 
Theorem \ref{Evrard's bound} and is the generalization of Polya's result in this 
setting.

  \begin{theorem}[\textbf{Generalized Property C}] \label{Bound for FD for 
polynomial of type (m,k)}
 Let $f$ be a primitive polynomial of type 
$(\mathbf{m},k)$, then $d(\underline{S},f)$ divides 
$\Gamma_{\mathbf{m}, k}(\underline{S})$ and this is sharp.
 \end{theorem}
 
 \begin{proof}From \cite{Cahen} (Prop. 
XI.1.9) on localization with respect 
to any nonzero prime ideal $\mathcal{P}$ of $R$, we have  
\begin{equation*}
(\mathrm{Int}_{\mathbf{m},k}(\underline{S},R))_{\mathcal{P}} = 
\mathrm{Int}_{\mathbf{m},k}(\underline{S},R_{\mathcal{P}}).
\end{equation*}
Thus the localization of 
$\Gamma_{\mathbf{m}, k}(\underline{S})$  at ${\mathcal{P}}$ is same as 
$\Gamma_{\mathbf{m}, k}(\underline{S})$ in $R_{\mathcal{P}}.$ 
 Hence, it suffices to prove the theorem in the case of a DVR $V$  with valuation $\nu$ and uniformizing parameter $\pi $ (i.e. $\nu(\pi)=1)$.

  Let $\Gamma_{\mathbf{m}, k}(\underline{S})= (\pi^s)$ for some $s \in \W$. Recall that for any polynomial $f \in K[\underline{x}]$ with coefficients $a_0, a_1, \ldots, a_n,\ \nu (f)$ is defined as $\operatorname{inf}_{0 \le i \le n} \nu (a_i).$  
Note that $\mathrm{Int}_{\mathbf{m},k}(\underline{S},V)$ is a   $V$-module finitely generated by polynomials $f_1,f_2,\ldots,f_r$ say (Prop. \ref{Int_basis} also gives a $V$-basis). By the definition of $\Gamma_{\mathbf{m}, k}(\underline{S}) $ these polynomials have to satisfy the condition  $$-\nu(f_j)\le \nu(\Gamma_{\mathbf{m}, 
k}(\underline{S})) =s \ \forall \ 0\le  j \le r,$$ and there must exist some $f_i$ such that $-\nu(f_i)= s,$ if not, then $\nu(\Gamma_{\mathbf{m}, k}(\underline{S}))$  will have valuation strictly less than $s$. It is clear that  the polynomial $\pi^s f_i \in V[\underline{x}]$ gives us sharpness.
 
   For a given primitive polynomial $f \in V[\underline{x}]$ of type $(\mathbf{m},k)$  with   
$\nu (d(\underline{S},f))=t, \tfrac{f}{\pi^t} $ belongs to $\mathrm{Int}_{\mathbf{m},k}(\underline{S},V) $ and
hence is a combination of $f_1,f_2,\ldots,f_r$. Consequently, $-\nu(\tfrac{f}{\pi^t} )$ cannot exceed  $s$. Since $\nu(f)=0$ we get $t \le s$ completing the proof.
 \end{proof}

 The following example  suggests that  the factorial defined by us gives a better bound for fixed 
 divisor than that of  \cite{Ev} and \cite{Bhar2} in some cases.

\begin{ex}\label{various bounds for fixed divisor}  If
    $f  $ 
    is a primitive polynomial of type $((2,2),3),$   we 
have the following bounds (refer Equations \eqref{Eqn:relation between bhargava and our factorial}   and \eqref{evrard} for their computations) for $d(\mathbb{Z} \times 2 \mathbb{Z},f):$
 \begin{enumerate}
 \item Theorem \ref{Evrard's bound} gives $3!_{ \mathbb{Z} \times 2 \mathbb{Z}}= 2^33!  $
 \item Corollary \ref{Bhargava's theorem} (or Theorem   \ref{Gun} ) gives  $2!_{ \mathbb{Z}} 2!_{2 
\mathbb{Z}}=2!2^22! $ 
 \item Theorem \ref{Bound for FD for 
polynomial of type (m,k)}  gives  $\Gamma_{(2,2),3}( \mathbb{Z} \times 2 \mathbb{Z})= 2^22!.$
 \end{enumerate}
Hence the polynomial  $\dfrac{f}{2^4}$ cannot be integer 
valued since $2^4$ exceeds $\Gamma_{(2,2),3}( \mathbb{Z} \times 2 \mathbb{Z})$. 
  We refer to the discussion after Corollary \ref{cor:fixed divisor of product of different variable 
polynomials} for a detailed comparison of these bounds.
 \end{ex}

  \medskip

   A multivariate   polynomial of  total  degree $k$ can be of different types and hence there are different bounds for its fixed divisor over any subset. For example, when $k=3$, a polynomial in two variables is one of the following types   
   
    \begin{center}
    \begin{tabular}{ccccc}
   $\{ ((3,0),3),((3,1),3),((3,2),3),((3,3),3),((2,1),3),$\\
  $ ((2,2),3),((2,3),3),((1,2),3),((1,3),3),((0,3),3)\}.$
    \end{tabular}
    \end{center} 

      Further taking $\underline{S}=  \mathbb{Z} \times 2 \mathbb{Z},$  by  Theorem \ref{Bound for FD for polynomial of type (m,k)} we have the following   bounds for fixed divisor  
 
 \begin{center}
\begin{tabular}{ |c|c|c|c|c|c|} 
 \hline
degree & (3,0) & (3,1) & (3,2) & (3,3) & (2,1)   \\ 
 \hline
 bound &  6   &  12     & 24  & 48 & 4     \\
 \hline
  degree &  (2,2)& (2,3) &(1,2) & (1,3) & (0,3)   \\ 
 \hline
 bound &   8   &   48    &  8  & 48 & 48    \\
 \hline
\end{tabular}
\end{center}

The generalized factorials also satisfy the following property.

\begin{pro}[\textbf{Generalized Property A}] For all  $\mathbf{m},\mathbf{m'} 
\in \W^n, k, k' \in \W$,
$\Gamma_{\mathbf{m}, k}(\underline{S}) \cdot \Gamma_{\mathbf{n}, 
k'}(\underline{S})$ divides $\Gamma_{\mathbf{m+n}, k+k'}(\underline{S})$.
\end{pro}
 
The proof follows by verifying the fact  
\begin{equation*}
\mathrm{Int}_{\mathbf{m},k}(\underline{S},R)\cdot 
\mathrm{Int}_{\mathbf{m'},k'}(\underline{S},R) \subseteq 
\mathrm{Int}_{\mathbf{m+m'},k+k'}(\underline{S},R).
\end{equation*}
\medskip

\section{$\nu_{\mathbf{m}}$-orderings and generalized factorials}\label{nu-ordering}
Now we shall introduce the concept of a $\nu_{\mathbf{m}}$-ordering of $\underline{S}$ which will 
help in establishing generalizations of Properties B and C with certain restrictions on $\underline{S}$. We will also obtain  
the local construction of $\Gamma_{\mathbf{m}, k}(\underline{S})$ using these 
$\nu_{\mathbf{m}}$-orderings. In this section we restrict to the case when $R=V$ where $V$ is a DVR with valuation $\nu$.

\medskip

 Let $K_{\mathbf{m}}[\underline{x}]$ be the vector subspace of $K[\underline{x}]$ containing polynomials of degree at most $\mathbf{m}$. Take the unitary monomial basis of  
$K_{\mathbf{m}}[\underline{x}]$  and place a total order on it  which 
is compatible with the total degree. Denote the cardinality of this basis by 
$l_{\mathbf{m}}$. Thus the monomials are arranged in a sequence $(p_j)_{0 \leq j 
< l_{\mathbf{m}}}$ with  $p_0=1$  and $\tdeg(p_i) \leq \tdeg(p_j)$ if $i<j$.  
For future reference, we also denote by $l_{\mathbf{m},k}$, the cardinality of 
the monomial basis of $\mathrm{Int}_{\mathbf{m},k}(\underline{S},V)$. Note that 
$l_{\mathbf{m},k} \leq \binom{n+k}{k}$. For any sequence of 
elements $\underline{ a}_0, \underline{a}_1,\ldots, \underline{a}_{r}$ in 
$V^n$ with $r<l_{\mathbf{m}}$, define

\begin{equation}\label{definition of new delta}
 \Delta_{\mathbf{m}}(\underline{a}_0, \underline{ a}_1, \underline{a}_2,\ldots, 
\underline{a}_{r})= \det(p_j(\underline{a_i}))_{0 \leq i, j \leq r}.
\end{equation}

\begin{definition} Let $\underline{S} \subseteq V^n$, a sequence of elements $\lbrace \underline{a}_i \rbrace_{0 \leq i < 
l_{\mathbf{m}} } $ of $\underline{S}$ is said to be a 
\textit{$\nu_{\mathbf{m}}$-ordering} of $\underline{S}$ if, for every $1 \leq r 
< l_{\mathbf{m}}$,
\begin{equation*}
\nu(\Delta_{\mathbf{m}}(\underline{a}_0, \underline{ a}_1, \ldots, 
\underline{a}_r))= 
\inf_{\underline{a} \in \underline{S}}\nu (\Delta_{\mathbf{m}}(\underline{a}_0, \ldots, 
\underline{a}_{r-1}, \underline{a})).
\end{equation*}
\end{definition}
The ordering defined above is the analogue of the $\nu$-ordering of \cite{Ev} 
mentioned in the introduction. All the important properties of these orderings 
occur only in the case that $\underline{S}$ is not contained in any algebraic 
subset of $K^n$. In other words, $I(\underline{S})=\lbrace f\in K[\underline{x}] : f(\underline{S})=0   \rbrace $  is trivial. This condition is a natural one to 
impose on $\underline{S}$ as $I(\underline{S})\neq \{0\}$ implies that 
$\Gamma_{\mathbf{m}, k}(\underline{S})=\{0\}$ for large enough values of $m_i$ 
and $k$.  {\em Hence we will assume the condition $I(\underline{S})= \{0\}$ for the 
rest of this paper.}

\medskip

We note that for any $\nu_{\mathbf{m}}$-ordering  $\lbrace \underline{a}_i 
\rbrace $ of $\underline{S}$, we have $\Delta_{\mathbf{m}}(\underline{a}_0, 
\underline{ a}_1, \ldots, \underline{a}_r) \neq 0$ for all $1 \leq r< 
l_{\mathbf{m}}$. This is because the vanishing of any of these would 
automatically give us a non-zero polynomial 
$ \Delta_{\mathbf{m}}(\underline{a}_0, \underline{ a}_1, \ldots, 
\underline{a}_t, \underline{x})$ which would belong to $I(\underline{S})$ 
contradicting our assumption on $\underline{S}$.

 Hence we can define the associated sequence of polynomials as follows
 \begin{definition} With all notations as above we define
$$F_{\mathbf{m},r}(\underline{x})= \dfrac{\Delta_{\mathbf{m}}(\underline{a}_0, 
\underline{ a}_1, \ldots, \underline{a}_{r-1},  
\underline{x})}{\Delta_{\mathbf{m}}(\underline{a}_0, \underline{ a}_1, \ldots, 
\underline{a}_r)} ,$$
for $1 \leq r < l_{\mathbf{m}}$ and $F_{\mathbf{0},0} (\underline{x})=1$.
\end{definition}

 Denote the vector space of all polynomials of degree at most $\mathbf{m}$ and total degree at most $k$ over $K$   by $K^{\lmk}[\underline{x}]$. The 
following result gives a criterion for membership in 
$\mathrm{Int}_{\mathbf{m},k}(\underline{S},V)$.

\begin{pro}\label{Int_basis} Given $\mathbf{m} \in \W^n, k \in \W$ and
$\underline{S} \subseteq V^n$,   let 
$\lbrace \underline{a}_i \rbrace $ be a
$\nu_{\mathbf{m}}$-ordering of $\underline{S}$. 
Then, the associated polynomials $\lbrace F_{\mathbf{m},r} \rbrace_{0 \leq r < 
l_{\mathbf{m},k}} $
 form a $V$-basis for the $V$-module 
$\mathrm{Int}_{\mathbf{m},k}(\underline{S},V)$. Hence, given $f \in 
V[\underline{x}]$ of type $(\mathbf{m}, k)$, we have the following
\begin{equation*}
f \in \mathrm{Int}_{\mathbf{m},k}(\underline{S},V) \Leftrightarrow 
f(\underline{a}_r) \in V  \ \mathrm{for} \  0\leq r < l_{\mathbf{m},k}.
\end{equation*}
\end{pro}
\begin{proof}
First note that $\lbrace F_{\mathbf{m},r} \rbrace_{0 \leq r < 
l_{\mathbf{m},k}} $ is a subset of 
$\mathrm{Int}_{\mathbf{m},k}(\underline{S},V)$ by the definition of 
$\nu_{\mathbf{m}}$-ordering. These polynomials also form a  basis of $K^{\lmk}[\underline{x}]$.
  Their expansion in terms 
of $\{p_r\}$ gives a lower-triangular matrix with the diagonal consisting of the 
non-zero entries
\begin{equation*}
\frac{\Delta_{\mathbf{m}}(\underline{a}_0, \underline{ a}_1, 
\underline{a}_2,\ldots, 
\underline{a}_{r-1})}{\Delta_{\mathbf{m}}(\underline{a}_0, \underline{ a}_1, 
\underline{a}_2,\ldots, 
\underline{a}_{r})} \cdot
\end{equation*}

Hence any $g \in \mathrm{Int}_{\mathbf{m},k}(\underline{S},V)$ can be expressed 
as $g(\underline{x}) = \sum_r c_r F_{\mathbf{m},r}(\underline{x})$ where $c_r \in K$.

Evaluation of  this expansion at $\underline{x} = \underline{a}_r$ for each $r$ 
gives us the matrix equation
\begin{equation*}
(g(\underline{a}_r))_{0 \leq r < l_{\mathbf{m},k}} = U \cdot (c_r)_{0 \leq r < 
l_{\mathbf{m},k}},
\end{equation*}
where $U$ is an upper-triangular matrix with entries in $V$ and unit   diagonal. 
Hence $U$ is  unimodular and has an inverse with entries in $V$, leading to the 
conclusion that $c_r \in V$ for all $r$. This establishes that $\lbrace F_{\mathbf{m},r} \rbrace_{0 \leq r < l_{\mathbf{m},k}} $ is a $V$-module basis for $\mathrm{Int}_{\mathbf{m},k}(\underline{S},V)$ .

The second statement follows by observing that we only need $g(\underline{a}_r) 
\in V$ in the above proof.
\end{proof}

Now we come to the next advantage of our approach over that of \cite{Ev}. 
Given a polynomial of type $(\mathbf{m}, k)$, it needs to be evaluated at $\lmk$ 
points in order to check for it being  an integer-valued, which in general, may be much 
smaller than the corresponding number $\binom{n+k}{k}$  if 
one considers only total degree (Cor. 17, \cite{Ev}). 
 \begin{ex} Consider the polynomial $$f(x,y)=1 - \dfrac{53 y}{30} + 
\dfrac{xy}{2}+\dfrac{12y^2}{5} -\dfrac{xy^2}{2}- \dfrac{19y^3}{30} \;\;and \;
\underline{S}= \Z_5 \times \Z_5.$$
If we will keep only total degree in mind (Cor. 17, \cite{Ev}) then 
corresponding to the monomial ordering $1,x,y,x^2,xy,y^2,x^3,x^2y,xy^2,y^3$  we must 
check   values of $f$ on first $\binom{3+2}{2}$ terms  
of $\nu$-ordering. The terms of $\nu$-ordering are - $(0,0), (1,0), (0,1), (2, 0), (1,1), (0, 2), (3, 0),   (2,1), (1, 2)$ and $(0, 3)$ with corresponding $f$ values  $1,1,1,1,1,2,1,1,1$  and 
$\dfrac{1}{5}$ respectively. The last value implies that this polynomial does not map
$\underline{S}$  back to  $\Z_5$.
  
   Now $\deg(f)=(1,3)$ and the monomial sequence is $1,x,y,xy,y^2,xy^2,y^3.$  Thus it is sufficient to check first $l_{(1,3),3}=7$  terms of $\nu_{(1,3)}$-ordering. The values of $f$ corresponding to  $\nu_{(1,3)}$-ordering $(0,0), (1,0), (0,1), (1,1),  (0, 2), (1, 2)$ and $(0, 3)$ are 
$1,1,1,1,2,1$ and $\dfrac{1}{5}$ respectively. Again the last value implies that polynomial 
doesn't maps $\underline{S}$  back to  $\Z_5$. 
 \end{ex}

Now we give the local construction of the new factorials. For that we define the 
following minor
\begin{equation*}
 \Delta_{\mathbf{m}}(s ; \underline{a}_0, \underline{ a}_1, 
\underline{a}_2,\ldots, 
\underline{a}_{r-1})= \det(p_j(\underline{a_i}))_{0 \leq i < r, 0 \leq j \leq r, 
j \neq s },
\end{equation*}
for $r<l_{\mathbf{m}}$ and $0 \leq s \leq r$.

Given the basis of $\mathrm{Int}_{\mathbf{m},k}(\underline{S},V)$  as in Proposition \ref{Int_basis}, 
the next corollary follows from the same argument as in Theorem  \ref{Bound for FD for 
polynomial of type (m,k)}.
\begin{cor}\label{local definition of gamma} Given $\mathbf{m}, k, \underline{S}, V, \nu$ and $\lbrace 
\underline{a}_i \rbrace $ as in Prop. \ref{Int_basis}, we have
\begin{equation}\label{local_defn}
\nu(\Gamma_{\mathbf{m}, k}(\underline{S}))= 
\max_{\substack{0\leq r < l_{\mathbf{m},k}\\ 0\leq s \leq r}} 
\nu 
\left(\frac{\Delta_{\mathbf{m}}(\underline{a}_0, \underline{ a}_1, \ldots, 
\underline{a}_r)}{\Delta_{\mathbf{m}}(s;\underline{a}_0, \underline{ a}_1, 
\ldots, \underline{a}_{r-1})}
\right).
\end{equation}
\end{cor}

 Note that this result implies that the right side of Equation 
\eqref{local_defn} is independent of the particular choice of 
$\nu_{\mathbf{m}}$- ordering. Conversely Equation \eqref{local_defn} can be used 
as a definition of the new factorial, provided we establish this independence by 
other means. One way to do that would be to first establish that the 
generalization of Property B holds for the new factorials, which is the last 
result of this section. Here we interpret property B as follows: the product 
$\prod_{i<j} (x_i - x_j)$ is the Vandermonde determinant $\det(f_j(x_i))$ 
where $f_j(x)$ is the monomial $x^j$; the product of factorials $0!1!\ldots n!$ 
is the particular value of this determinant for the choice $x_i = i$ which plays 
the role of the $\nu $ - ordering for a valuation $\nu$ coming from 
any prime ideal of $\Z$.

\begin{pro}[\textbf{Generalized Property B}]\label{prop_B} Given $\mathbf{m}, 
\underline{S}, V, \nu$ and $\lbrace \underline{a}_i \rbrace $ as in Prop. 
\ref{Int_basis}, we have for $r<l_{\mathbf{m}}$
\begin{equation*}
\nu(\Delta_{\mathbf{m}}(\underline{a}_0, \underline{ a}_1, \ldots, 
\underline{a}_r))= 
\min_{\underline{x}_0, \underline{ x}_1, \ldots, \underline{x}_r \in 
\underline{S} }
\nu(\Delta_{\mathbf{m}}(\underline{x}_0, \underline{ x}_1, \ldots, 
\underline{x}_r)).
\end{equation*}
\end{pro} 
\begin{proof}

Let $r\le l_{\mathbf{m}}$ be fixed then we know $(p_j)_{0 \leq j \leq r}$ generates the same vector space over $K$ as $(F_{\mathbf{m},i} )_{0 \leq i \leq r} $ generates.  Denote the change of basis matrix by $M$ then we have

\begin{equation}\label{eq; det1}   
  \mathrm{det}(p_j(\underline{ a}_i))= \mathrm{det}(M)\mathrm{det}(F_{\mathbf{m},j}(\underline{ a}_i)). 
\end{equation}

Let $\underline{x}_0, \underline{ x}_1, \ldots, 
\underline{x}_r$ be an arbitrary sequence of elements of $\underline{S}$ then
\begin{equation}\label{eq; det2} 
 \mathrm{det}(p_j(\underline{x}_i))= \mathrm{det}(M)\mathrm{det}(F_{\mathbf{m},j}(\underline{x}_i)).
\end{equation}

By substracting       Equation (\ref{eq; det1}) from Equation (\ref{eq; det2}) after taking valuation   we get
$$ \nu (\mathrm{det}(p_j(\underline{x}_i)))-\nu (\mathrm{det}(p_j(\underline{ a}_i)))= \nu (\mathrm{det}(F_{\mathbf{m},j}(\underline{x}_i)))-\nu (
\mathrm{det}(F_{\mathbf{m},j}(\underline{ a}_i))).$$
 Since $(F_{\mathbf{m},j})_{j \geq 0}$  are integer valued and $ \mathrm{det}(F_{\mathbf{m},j}(\underline{ a}_i))=1$, $ \nu (\mathrm{det}(p_j(\underline{x}_i))) \geq \nu (\mathrm{det}(p_j(\underline{ a}_i)))$ completing the proof.
\end{proof}

It follows from this result that the sequence 
$\nu_{\mathbf{m}}(\Delta(\underline{a}_0, \underline{ a}_1, \ldots, 
\underline{a}_r))$ is independent of the choice of the particular 
$\nu_{\mathbf{m}}$-ordering.

 \section{Fixed divisor in the case of Cartesian product of sets}{\label{s3}} This 
section is devoted to the case when $\underline{S}=S_1 \times 
S_2 \times \cdots \times S_n$ where each $S_i \subseteq R$.  We start this 
section by fixing few notations. For any $n$-tuple $(i_1,i_2,\ldots,i_n)= 
\mathbf{i}$, its sum of  components will be denoted by $\vert \mathbf{i} \vert$ and $\mathbf{i}!_{\underline{S}}$ will   denote $i_1!_{S_1} \ldots i_n!_{S_n}.$

\begin{pro}\label{Pro:relation between bhargava and our factorial} In the case when $\underline{S}=S_1 \times 
S_2 \times \cdots \times S_n $, we have
  \begin{equation}\label{Eqn:relation between bhargava and our factorial}
 \Gamma_{\mathbf{m}, k}(\underline{S})  
  =\operatorname*{lcm}\limits_{\substack{\mathbf{0}  \leq \mathbf{i} \leq 
\mathbf{m} \\ \vert \mathbf{i} \vert \leq k}}\mathbf{i}!_{\underline{S}}.
\end{equation}
\end{pro}

\begin{proof}

 It suffices to prove in the case when $R$ is a DVR. Let $S \subseteq V$ be a 
non-empty subset of the DVR $V$ and $\lbrace a_i \rbrace_{i \geq 0}$ be some 
$\nu$-ordering of $S.$   Define
 $$\binom{x}{r}_S=  
\dfrac{(x-a_0)(x-a_1)\ldots(x-a_{r-1})}{(a_r-a_0)(a_r-a_1)\ldots(a_r-a_{r-1})}.$$ The denominator is clearly $r!_{S}$.  In our setting 
$\underline{S}=S_1 \times S_2 \times \cdots \times S_n$ with some choice of $\nu$-ordering for $S_j$'s   we can  define analogously for $\mathbf{i} \in \mathbb{W}^n$
\begin{equation}\label{basis of Int(S,V)}
\binom{\underline{x}}{\mathbf{i}}_{\underline{S}}=\binom{x_1}{i_1}_{S_1} 
\binom{x_2}{i_2}_{S_2} \cdots \binom{x_n}{i_n}_{S_n}.
\end{equation}
These polynomials form a $V$-module basis for 
Int$_{\mathbf{m},k}(\underline{S},V)$ provided we consider only those 
$\mathbf{i}$'s having the properties $\mathbf{i} \leq \mathbf{m}$ and $\vert \mathbf{i} \vert \leq k$ 
simultaneously as in Proposition \ref{Int_basis}.  As we 
pointed out in the proof of  Theorem \ref{Bound for FD for 
polynomial of type (m,k)},    $\nu ( \Gamma_{\mathbf{m}, k}(\underline{S}))$ will be the maximum of the valuations of the denominators of this basis, i.e.,

 $$\nu ( \Gamma_{\mathbf{m}, k}(\underline{S}))= \operatorname*{max} 
\limits_{\substack{\mathbf{0}  \leq \mathbf{i} \leq \mathbf{m} \\ \vert \mathbf{i} \vert \leq 
k}}\nu ( \mathbf{i}!_{\underline{S}}).$$

\end{proof}

For any subset $T\subseteq R$ and any ideal $I$ with prime factorization 
$ I=\prod_{i=1}^r P_i^{e_i}$, define an $\textbf{$I$-ordering}$ of $T$ to be a 
sequence 
$\{ a_j \}_{j=0}^{\infty}$ in $R$ which is congruent modulo $P_i^{e_i +1}$ to a 
$P_i$-ordering of $T$ for each $i=1,2,\ldots, r$. This type of sequence was also 
constructed by 
Bhargava (see \cite{Bhar1}, Sec. 3) in case of quotient of Dedekind domain.

Now, for our setting of $\underline{S}=S_1 \times S_2 
\times \cdots \times S_n$ and $I$ any fixed ideal, let $\{a_{i,j} \}_{i=0}^{\infty}$ 
be an $I$-ordering of $S_j$ for each $j=1,2,\ldots,n$. Given 
$\mathbf{i}=(i_1,i_2,\ldots,i_n) \in \mathbb{W}^n$, let 
$a_{\mathbf{i}}=(a_{i_1,1},a_{i_2,2},\ldots,a_{i_n,n})$ and the associated 
polynomial 
$B_{\mathbf{i}}(\underline{x}) = \prod_{j=1}^{n} 
\prod_{k=0}^{i_j-1}(x_j-a_{k,j})$. For a given prime ideal $P$ and ideal $J$ of $R$, $w_P(J)$ will denote the highest power of $P$ dividing $J$. With these notations the  following lemma is straightforward.

\begin{lemma}\label{lem2}
For all $\mathbf{i} \in \mathbb{W}^n$ such that 
$\Gamma_{\mathbf{i},\vert \mathbf{i} \vert}(\underline{S})$ divides $I$ we have
\begin{itemize}
	\item[(i)] For all $\mathbf{k}\in \mathbb{W}^n$ such that some 
component $k_j < i_j$, $B_{\mathbf{i}}(a_{\mathbf{k}})=0$;
	\item[(ii)] For all prime $P$ dividing $I$ and $\underline{s} \in 
\underline{S}$, $w_{P}(\Gamma_{\mathbf{i},\vert \mathbf{i} \vert}(\underline{S}))= 
w_{P}(B_{\mathbf{i}}(a_{\mathbf{i}}))$ and $ 
w_{P}(B_{\mathbf{i}}(a_{\mathbf{i}}))\ | \ 
w_{P}(B_{\mathbf{i}}(\underline{s}))$. 
\end{itemize}
\end{lemma}

Let $\lbrace  p_j \rbrace_{j \ge 0}$ be the monomials in $K_{\mathbf{m}}[\underline{x}]$ ordered in a sequence compatible with total degree (see the paragraph before Equation \ref{definition of new delta}). Given $f \in R[\underline{x}]$ of 
type $(\mathbf{m},k)$, we have
$$f(\underline{x})=\sum\limits_{j=0}^{\lmk -1}
c_jp_j(\underline{x}),$$ 
where all coefficients $c_j \in R$. We denote the degree of the last monomial in the above expression by $\mathbf{k}$.
We now take $I = \operatorname*{lcm}( d(\underline{S},f),\Gamma_{\mathbf{m},k}(\underline{S}))=\prod_{i=1}^r P_i^{e_i}$ 
and construct $a_{\mathbf{i}}=(a_{i_1,1},a_{i_2,2},\ldots,a_{i_n,n})$ as 
described above.  It can be seen that $f$ 
also has the following representation
 
 \begin{equation}\label{e2}
 f(\underline{x})= 
\sum\limits_{\substack{\mathbf{0}  \leq \mathbf{i} \leq 
\mathbf{m} \\    \vert \mathbf{i} \vert \le 
k}}b(\mathbf{i})B_{\mathbf{i}}(\underline{x})
.  \end{equation}
  We write this expression in such a way that it ends with $B_{\mathbf{k}}(\underline{x}).$ Now we present the main theorem of this section which can be viewed as a generalization of Theorem \ref{Pol} in this setting.

\begin{theorem}{\label{t2}} Let $f$ be a primitive polynomial of type $(\mathbf{m},k)$ 
and $b(\mathbf{i})$ be as in \eqref{e2}.
Then 
\begin{equation}\label{dsf}
 d(\underline{S},f) 
=(b({\mathbf{0}})\Gamma_{\mathbf{0},0}(\underline{S}), \ldots, 
b({\mathbf{i}})\Gamma_{\mathbf{i},\vert \mathbf{i} \vert}(\underline{S}), \ldots,  
b({\mathbf{k}})\Gamma_{\mathbf{k},\vert \mathbf{k} \vert}(\underline{S})).
\end{equation}
Consequently, $d(\underline{S},f)$ divides $\Gamma_{\mathbf{m},k}(\underline{S})$ and this is sharp. Conversely, for each $I$ dividing $\Gamma_{\mathbf{m},k}(\underline{S})$ there exists a primitive polynomial $f$ of type $(\mathbf{m},k)$ with $d(\underline{S},f)=I.$
  \end{theorem}
\begin{proof}    Let $P_j$ be any prime ideal dividing $d(\underline{S},f)$ and 
$P_j^e=w_{P_j}(d(\underline{S},f))$. Then, by construction 
$f(a_{\mathbf{i}}) \equiv f(\underline{s})$ modulo $P_j^{e_j+1}$ for some 
$\underline{s} \in \underline{S}$ and hence
$P_j^e$  divides  $f(a_{\mathbf{i}})$.  
We claim that $P_j^e$ divides $b(\mathbf{i})\Gamma_{\mathbf{i},\vert \mathbf{i} \vert}(\underline{S})$ 
and 
establish it by 
induction on $\vert \mathbf{i} \vert$ as follows. 

\medskip

The base case is clear from the observation that $f(a_{\mathbf{0}}) = 
b({\mathbf{0}}) \Gamma_{\mathbf{0},0}(\underline{S})$. Let 
induction hypothesis be true for all $\mathbf{i}$ for which $\vert \mathbf{i} \vert  
 \le r$. Let $\mathbf{j}$ be an arbitrary index 
such that $\vert\mathbf{j} \vert=r+1$. Consider the expansion \eqref{e2} of 
$f(a_{\mathbf{j}})$. By Lemma \ref{lem2}(i), the sum is over the indices 
$\mathbf{i}\leq \mathbf{j}$. All of these indices, excluding $\mathbf{j}$, have 
sum of components less than or equal to $r$. Hence by Lemma \ref{lem2}(ii) and 
the induction hypothesis, we get the desired result that
$P_j^e$ divides $b(\mathbf{j}) \Gamma_{\mathbf{j},\vert \mathbf{j} \vert}(\underline{S})$. This 
establishes the claim. Consequently  $P_j^e$ and hence $d(\underline{S},f)$ divides 
$(b({\mathbf{0}})\Gamma_{\mathbf{0},\vert \mathbf{0} \vert}(\underline{S}), \ldots, 
b({\mathbf{i}})\Gamma_{\mathbf{i},\vert \mathbf{i} \vert}(\underline{S}), \ldots,  
b({\mathbf{k}})\Gamma_{\mathbf{k},\vert \mathbf{k} \vert}(\underline{S}))$.

In the other direction  $(b({\mathbf{0}})\Gamma_{\mathbf{0},\vert \mathbf{0} \vert}(\underline{S}), \ldots, 
b({\mathbf{i}})\Gamma_{\mathbf{i},\vert \mathbf{i} \vert}(\underline{S}), \ldots,  
b({\mathbf{k}})\Gamma_{\mathbf{k},\vert \mathbf{k} \vert}(\underline{S}))$  divides 
$f(\underline{s})$ for 
all $\underline{s} \in \underline{S}$ (by Lemma \ref{lem2}(ii)) and hence 
divides 
$d(\underline{S},f)$ too. This establishes \eqref{dsf}.

\medskip

Note that $(b(\mathbf{0}),\ldots, b(\mathbf{k}))= (c(\mathbf{0}),\ldots, 
c(\mathbf{k}))$ due to the unimodularity of the matrix which transforms one set 
of coefficients to the other. 
So, when $f$ is primitive and $P$ divides $d(\underline{S},f)$, there exists 
$\mathbf{i}$ such that $P$ does not divide $b(\mathbf{i})$. 
Then $w_P(d(\underline{S},f))$ divides $\Gamma_{\mathbf{i},\vert \mathbf{i} \vert}(\underline{S})$. 
Since 
$\mathbf{i} \leq \mathbf{m}$ and $\vert \mathbf{i} \vert \le k$, 
 $\Gamma_{\mathbf{i},\vert \mathbf{i} \vert}(\underline{S})$  must divide 
$\Gamma_{\mathbf{m},k}(\underline{S})$, which gives the desired result. For 
every ideal $I$ dividing $\Gamma_{\mathbf{m},k}(\underline{S})$ selection of 
$b(\mathbf{i})$'s suitably  will give us  a primitive polynomial $f$ such that 
$d(\underline{S},f)=I$. This proves  sharpness also. 
\end{proof}
  
   Relaxing the condition of total degree $k$ in  Theorem \ref{t2} and  using Proposition \ref{Pro:relation between bhargava and our factorial} we get
 \begin{cor}[Bhargava \cite{Bhar2}] \label{Bhargava's theorem}
 Let $f \in R[\underline{x}]$ be a primitive polynomial of degree $\mathbf{m}$ and $b(\mathbf{i})$ be as in \eqref{e2} (with appropriate restrictions on $\mathbf{i}$). Then
\begin{equation}\label{dsf of Bhargava in multivariate case}
 d(\underline{S},f) 
=(b({\mathbf{0}}) \mathbf{0}!_{\underline{S}}, \ldots, 
b({\mathbf{i}})\mathbf{i}!_{\underline{S}}, \ldots,  
b({\mathbf{m}})\mathbf{m}!_{\underline{S}}).
\end{equation}
   Hence  $d(\underline{S},f)$ divides $ \mathbf{m}!_{\underline{S}}$ and 
this is sharp. Conversely, for each $I$ dividing $ \mathbf{m}!_{\underline{S}}$ there exists a primitive polynomial $f$ of degree $\mathbf{m}$ with $d(\underline{S},f)=I.$
   \end{cor}
  
   The following corollary shows the behaviour of the fixed divisor of a multivariate separable  
polynomial. Its proof follows by induction on the number of variables and 
by Theorem \ref{t2}.

\begin{cor}\label{cor:fixed divisor of product of different variable 
polynomials} Let $f_i(x_i) \in R[x_i]$ for  $ 1 \le i \le k.$  Then

$$d(S_1 \times S_2 \ldots \times S_k ,f_1f_2 \ldots f_k) =d(S_1,f_1)d(S_2,f_2)\ldots 
d(S_k,f_k).$$

\end{cor}
  
 We close this section  by comparing various bounds of fixed divisor. Recall (see \cite{Ev}, 
Example 3) that the generalized factorial $k!_{\underline{S}}$ when 
$\underline{S}$ is a Cartesian product is given by 
\begin{equation} \label{evrard}
k!_{\underline{S}} = \operatorname*{lcm}\limits_{ \vert \mathbf{i} \vert = 
k}\mathbf{i}!_{\underline{S}}.
\end{equation} 
 For a given primitive polynomial $f$ of type
$(\mathbf{m},k)$, Corollary \ref{Bhargava's theorem} and Theorem \ref{Evrard's 
bound}
give different bounds on the fixed divisor, viz. 
$ \mathbf{m}!_{\underline{S}}$ and $k!_{\underline{S}}$ respectively and 
these are not comparable in general. Depending upon the values of $\mathbf{m}$, 
$k$ and the nature of the subsets $S_i$, any one result might be stronger than 
the other.

For example, let $\underline{S}= \mathbb{Z} \times \mathbb{Z}$  and $f$ be a 
polynomial with integer coefficients
with degree $(5,5)$. If the total degree is $10$ (for e.g., $f(x,y)= x^5y^5$) 
then Theorem \ref{t2} asserts that its fixed divisor will divide $5!5!$ whereas 
Theorem \ref{Evrard's bound} asserts that it will divide $10!$. In this case the 
former is 
stronger than the latter.
On the other hand, if the total degree of the polynomial $f$ is  $5$ (for e.g., 
$f(x,y)=x^5+y^5$) then Theorem \ref{t2} still says that its fixed divisor will 
divide $5!5!$ whereas Theorem \ref{Evrard's bound} says that it will divide 
$5!$. In this 
case the latter is stronger. 

Now we note that our factorial always gives a stronger result. If  $f(x,y)= 
x^5y^5 $  then $d(\underline{S},f)$ divides 
$\Gamma_{(5,5),10}(\underline{S})=5!5!$ and if $f(x,y)=x^5+y^5$  
 then $d(\underline{S},f)$ divides 
$\Gamma_{(5,5),5}(\underline{S})=5!.$ Thus in both the cases we get a better bound and this is not a coincidence!  $\Gamma_{\mathbf{m},k}(\underline{S})$ always divides $k!_{\underline{S}}$ and $\mathbf{m}!_{\underline{S}}$ but need not to be equal to their g.c.d. as Example \ref{various bounds for fixed divisor} suggests.
 
\section{Formula for fixed divisor in the general case}{\label{s4}}

In this section we look for various formulae for $d(\underline{S},f)$ when 
$\underline{S}$ is an arbitrary subset of $R^n$ such that $I(\underline{S})=\{0\}$ 
and $R$ is a Dedekind domain 
with field of fractions $K$. We start with few notations; $f(\underline{x})$ 
will denote a primitive polynomial of type $(\mathbf{m},k)$ and  $\mathbb{P}= 
\lbrace P_1, P_2,\ldots, P_r \rbrace$ will denote the set of all prime ideals of 
$R$ which appear in the prime factorization of $\Gamma_{\mathbf{m}, 
k}(\underline{S})$.
For each prime ideal $P_i \in \mathbb{P}$, the localization $R_{P_i}$ is a DVR 
with valuation $\nu_i$, say. For $j=1,2,\ldots,r$, let 
$\{\underline{a}_{i,j}\}_{i=0}^{\lmk-1}$ be a $\nu_{\mathbf{m},j}$-ordering 
(i.e., $\nu_{\mathbf{m}}$-ordering corresponding to $\nu_j$) of $\underline{S}$ 
and ${e_j}= \nu_j \left(\Delta_{\mathbf{m}}(\underline{a}_{0,j}, \underline{ 
a}_{1,j},\ldots, \underline{a}_{\lmk-1,j})\right)$. Now consider a sequence 
$\{\underline{a}_i\}_{i=0}^{\lmk-1}$ in $R^n$ which satisfies the congruences 
\begin{equation}{\label{e3}}
\underline{a}_i \equiv \underline{a}_{i,j}\ \mathrm{mod}\ P_j^{e_j + 1}.
\end{equation}

Here $\underline{x} \equiv \underline{y}\ \mathrm{mod}\ I$ denotes $x_i 
\equiv y_i\ \mathrm{mod}\ I$ for all the components $1 \leq i \leq n$. We will 
define the polynomials $B_0(\underline{x})=1$ and $B_j(\underline{x})= 
\Delta_{\mathbf{m}}(\underline{a}_0, \underline{ 
a}_1,\ldots,\underline{a}_{j-1}, 
\underline{x})$ for $1\leq j < \lmk$. The following lemmas are  easy to prove and hence we omit the proofs.

\begin{lemma}{\label{l1}} For $1\leq j < \lmk$, we have
\begin{itemize}
	\item[(i)] For every $P_i \in \mathbb{P}$, and $\underline{s} \in 
\underline{S}$, $\nu_i(B_{j}(\underline{a}_j)) \le 
\nu_i(B_{j}(\underline{s}))$;
	\item[(ii)] For $0\leq m<j$, $B_{j}(\underline{a}_{m})=0$. 
\end{itemize}
\end{lemma}	

 \begin{lemma}{\label{l2}} Let $P_i \in \mathbb{P}$ with $P_i^e = 
w_P(d(\underline{S}, f))$, then $P_i^e$ divides 
$f(\underline{a}_j)$ for all $0\leq j < \lmk$. 
 \end{lemma}

Let $\mathbb{T}$ be a finite set of non-zero prime ideals of $R$. For a  given ideal $I \subset R$ define
\begin{equation*}
I_{\mathbb{T}} = \prod_{P \in \mathbb{T}} w_P(I).
\end{equation*}
 For example, let $R=\mathbb{Z}$ and $\mathbb{T}=\{2\mathbb{Z},3\mathbb{Z}\}$ then $2^2 3^2 
5^37
\mathbb{Z}_{\mathbb{T}}=2^2 3^2\mathbb{Z}$. 
 \medskip

Now, we give a formula for the fixed divisor in general setting.
\begin{pro}\label{t4} Let $f(\underline{x})$ be a primitive polynomial of 
type $(\mathbf{m},k)$, then there exist $b_0, b_1, 
\ldots, b_{\lmk-1}$ in $K$ such 
that 
\begin{equation}\label{dsf2}
d(\underline{S}, f)=(b_0, b_1\Delta_{\mathbf{m}}(\underline{a}_0, 
\underline{a}_1), \ldots, b_{\lmk-1}\Delta_{\mathbf{m}}(\underline{a}_0, 
\underline{ 
a}_1,\ldots,\underline{a}_{\lmk-1}))_{\mathbb{P}}.
\end{equation}
\end{pro}
 \begin{proof} Clearly, there exist  $b_0, b_1, \ldots, b_{\lmk-1}$ in 
$K$ 
such that 
\begin{equation}\label{expansion}
f(\underline{x})=\sum_{0 \leq i < \lmk} b_i B_i(\underline{x}).
\end{equation}

 Let 
$P$ be a prime dividing $d(\underline{S}, f)$ and $P^e = w_P(d(\underline{S}, 
f))$.
Then by Lemma \ref{l2}, 
$P^e$ must divide $f(\underline{a}_i)$ for $0\leq i < \lmk$. 
By substituting $\underline{x}= \underline{a}_i$ in \eqref{expansion} 
inductively, we 
see that $P^e$ divides each fractional ideal generated by
 $b_i 
\Delta_{\mathbf{m}}(\underline{a}_0, \underline{ a}_1, \underline{a}_2,\ldots, 
\underline{a}_i)$ and so divides the right side of \eqref{dsf2}.
 
 Conversely, if $P \in \mathbb{P}$ be any ideal such that $P^e$ divides each 
fractional ideal
$\sloppy{b_i \Delta_{\mathbf{m}}(\underline{a}_0, \underline{ a}_1, 
\underline{a}_2,\ldots, 
\underline{a}_i)}$ for $0\leq i < \lmk$, then $P^e$ divides 
$f(\underline{s})$ for all $\underline{s} \in \underline{S}$ by 
\eqref{expansion} and Lemma \ref{l1}. This completes the proof of the theorem.
 \end{proof}
 
  Note that the matrix which transforms the coefficients $c_i$ in the usual 
representation $f(\underline{x})=\sum_{0\leq i<\lmk} c_i 
p_i(\underline{x}) $ to the coefficients $b_i$ in \eqref{expansion} can be 
computed  by first expanding $B_i(\underline{x})$ into 
monomials and then finding the inverse of the appropriate matrix.

\medskip
The following result  is the converse of Theorem \ref{Bound for FD for polynomial 
of type (m,k)}.
 
\begin{pro}\label{t5}
Let $I$ be any divisor of $\Gamma_{\mathbf{m}, k}(\underline{S})$, then there 
exists a primitive 
polynomial $f$ of type $(\mathbf{m},k)$ and degree  such that $d(\underline{S}, 
f)=I$.
\end{pro}

\begin{proof} By Theorem \ref{Bound for FD for polynomial of type (m,k)}, there 
exists a primitive polynomial $g$ of type
 $(\mathbf{m},k)$ with
$g(\underline{x})= \sum\limits_{0\leq i<\lmk} c_i p_i(\underline{x})$   such that $d(\underline{S}, g)= \Gamma_{\mathbf{m}, 
k}(\underline{S})$. 
Recall that the set of primes dividing 
$\Gamma_{\mathbf{m}, k}(\underline{S})$ is $\mathbb{P}=\{P_1,P_2,\ldots,P_r\}$. Let $\{Q_1,Q_2,\ldots,Q_s\}$ be the set 
of primes dividing 
$(c_1,c_2,\ldots,c_{l-1})$. We note that these two sets have no intersection. 
For, if not, let $Q_i \in \mathbb{P}$, then $Q_i$ divides $g(\underline{s}) 
\equiv c_0 \ \mathrm{ mod }\ Q_i$ for some $\underline{s} \in \underline{S}$. 
This means that $Q_i$ divides $c_0$ and $g$ is not primitive, which is a 
contradiction.

 Choose $b \in R$ such that $w_{P_i}(\langle b 
\rangle)=w_{P_i}(I)$ for all $i=1,2,\ldots,r$ and $Q_i \ | \ b$ for 
$i=1,2,\ldots,s$.

Consider the polynomial $f=b+g$. Clearly, $d(\underline{S}, f)=I$ and $f$ is 
primitive.
\end{proof}

 The following result is the analogue of Hensel's result (Theorem~\ref{Hensel}).
\begin{theorem}\label{t6}
Let $f$ be a polynomial of type $(\mathbf{m},k)$ 
and $\underline{a} \in \underline{S}$ be such that $f(\underline{a})\neq 0$. 
Then there exist elements $\underline{a}_1, \ldots, \underline{a}_{\lmk-1}$ in 
$R^n$ 
such that $d(\underline{S},f)$ is given by $$ 
d(\underline{S},f)=(f(\underline{a}),f(\underline{a}_1),\ldots,f(\underline{a}_{
\lmk-1})).$$
\end{theorem}

\begin{proof}
 Consider the prime factorization $\langle f(\underline{a}) \rangle= 
\prod_{i=0}^r P^{e_i}_i$. Now we construct a sequence $\lbrace \underline{a}_j 
\rbrace _{0\leq j < \lmk}$ which is term-wise congruent to a $\nu_i$-ordering of 
$\underline{S}$ modulo $P^{e_i+1}_i$.   For each $\nu_i$-ordering, 
we put the condition that the first element is $\underline{a}$.  Hence, we can 
assume that $\underline{a}_0 = \underline{a} \in \underline{S}$.

 It can be shown, as in Lemma \ref{l2}, that if $P^{e}$ divides 
$d(\underline{S},f)$, then $P=P_i$ for some $i$ and $P^{e}$ divides 
$f(\underline{a}_j)$ for every $0 \leq j < \lmk$. 

In the other direction, we express $f(\underline{x})=\sum_{0\leq j < \lmk} b_j
B_j(\underline{x})$ (these $B_j$ are defined as before in terms of $a_j$).  Now 
if $P^{e}$ divides 
$(f(\underline{a}_0),f(\underline{a}_1),\ldots,f(\underline{a}_{l-1}))$, then 
it 
must divide $b_j B_j(\underline{a}_j)$ for $0\leq j < \lmk$ (by induction) and 
so 
it must divide $f(\underline{s})$ for all $\underline{s} \in \underline{S}$ (as shown 
in Lemma \ref{l1}).
\end{proof}

In the case when $\underline{S}$ is a Cartesian product of subsets of $R$ and 
$f$ a polynomial of degree $\mathbf{m}$, the elements $a_{\mathbf{i}}$ can be 
constructed as in Section \ref{s3} from a $\langle 
f(\underline{a}_0)\rangle$-ordering in each component, starting with arbitrary 
$\underline{a}_0 \in \underline{S}$ such that $f(\underline{a}_0) \neq 0$. Then 
$d(\underline{S},f)$ can be shown to be the g.c.d. of the $f$ images of $\lmk$ 
elements (which is always less than or equal to the bound $(m_1+1)(m_2+1)\cdots (m_n+1)$ given by Theorem
\ref{Hensel}), which might be more 
useful than Theorem \ref{t2} in certain situations.

For example, given all assumptions of Theorem \ref{Hensel} where 
$f \in \mathbb{Z}[\underline{x}]$ is a polynomial with degree $m_i$ in $x_i$ for 
$i=1,2,\ldots , n$. Then 
\begin{equation}\label{Hensel example}
d(\Z^n,f) =\operatorname*{gcd} \lbrace  f(r_1,r_2, \ldots , r_n) : 0 \leq r_i 
\leq m_i \rbrace.
\end{equation}

Further  if  $\tdeg(f)=k$ then by Theorem \ref{t6} we have
\begin{equation}\label{Hensel example improved}
d(\Z^n,f) =\operatorname*{gcd} \lbrace  f(r_1,r_2, \ldots , r_n) : 0 \leq r_i 
\leq m_i, r_1+r_2+ \ldots + r_n \leq k \rbrace.
\end{equation}

From Equations ~\eqref{Hensel example} and  ~\eqref{Hensel example improved},  $d(\Z^n,f)$ can be evaluated by finding the g.c.d. of finite number of images of $f$. Further Equation ~\eqref{Hensel example improved} uses less number of $f$ images than that of Equation ~\eqref{Hensel example}.

It is well known that every ideal in a Dedekind domain is generated by   two elements. The following result shows that for   $d(\underline{S},f)$, those elements can be taken from images of $f$.

\begin{theorem}\label{fixed divisor by two image elements}
Let $f(\underline{x}) \in R[\underline{x}]$ be  a polynomial of type 
$(\mathbf{m},k)$, then for each 
element $\underline{a} \in \underline{S} \subseteq R^n$ such that 
$f(\underline{a}) \neq 0$, there exists an element $\underline{b} \in R^n$ such 
that $d(\underline{S},f)=(f(\underline{a}),f(\underline{b})).$
\end{theorem}

\begin{proof} Let $\underline{a} \in \underline{S} \subseteq R^n$ such that 
$f(\underline{a}) \neq 0$  and $\Pi_{i=0}^rP_i^{e_i}$ be the prime 
factorization of $\langle f(\underline{a})\rangle$. For each prime $P_i$ we find an element 
$b_{i,r_i}$ among first $\lmk-1$ terms 
of $\nu_i$-ordering of $\underline{S}$ such that $f(\underline{b}_{i,r_i})$ is 
divisible by the smallest power of $P_i$. Now we select $\underline{b}$ which is congruent 
to $\underline{b}_{i,r_i}$ modulo a sufficiently high power of  $P_i$ for all $0\le 
i\le r.$ Then it is easy to check that  
$d(\underline{S},f)=(f(\underline{a}),f(\underline{b}))$.

\end{proof}

  In general, $d(\underline{S},f)$ may not be generated by a single $f(a)$ for some $a \in R$. For example, if $f=5x+3$, then $d(\Z,f)= \Z,$ but one cannot find $m \in \Z$ such that $\langle f(m)\rangle=\Z.$

\medskip

The following corollary gives a relation connecting $d(\underline{S}, fg)$, $d(\underline{S}, f)$ and $d(\underline{S}, g)$. Its proof follows from Theorem  \ref{t6}.
 \begin{cor} \label{cor:fixed divisor of product}
  Let $f(\underline{x})$and $g(\underline{x})$ be two primitive polynomials of 
type $(\mathbf{m}_1,k_1)$ and $(\mathbf{m}_2,k_2)$. If 
$\mathbf{m}=\mathbf{m}_1+\mathbf{m}_2$ and $k=k_1+k_2$, then there exist  
elements $\underline{a}_0, 
\underline{a}_1, \ldots, \underline{a}_{\lmk-1}$
in $R^n$ such that 
 $$d(\underline{S}, fg)= (f(\underline{a}_0)g(\underline{a}_0), 
f(\underline{a}_1)g(\underline{a}_1), \ldots, 
f(\underline{a}_{\lmk-1})g(\underline{a}_{\lmk-1})),$$ where $$d(\underline{S}, 
f)= 
(f(\underline{a}_0), f(\underline{a}_1), \ldots, f(\underline{a}_{\lmk-1}))$$ 
and 
$$d(\underline{S}, g)= (g(\underline{a}_0), g(\underline{a}_1), \ldots, 
g(\underline{a}_{\lmk-1})).$$
 \end{cor}
 
 \medskip

 The polynomials satisfying  $d(\underline{S}, fg)= d(\underline{S}, f)d(\underline{S},g)$ are closely related to irreducibility in Int$(\underline{S},\Z)$ where $\underline{S} \subseteq \Z$ (see  \cite{Chapman}, Theorem 2.8 ).
Corollary ~\ref{cor:fixed divisor of product} may be useful in that direction.

The following result, proved in the case of a DVR in \cite{Bhar2}, can be 
derived in the single variable case by Theorem \ref{t6}. 
\begin{cor}\label{cor: fixed divisor in case of simultaneous P-ordering} Let 
$S$ 
be a subset of $R$ that admits a simultaneous $P$-ordering,  i.e., a sequence 
$\{a_i\}$ in $S$ which is a $P$-ordering of $S$ for all non-zero primes $P$ and 
$f \in R[x]$ a polynomial of degree $k$. Then 
\begin{equation*}
d(S,f) = (f(a_0),f(a_1),\ldots,f(a_k)).
\end{equation*} 
\end{cor}
For some examples of subsets with simultaneous $P$-orderings, see \cite{Bhar3}, 
\cite{Adam1} and \cite{Adam2}.   

To conclude, we would like to remark that Theorems \ref{t6} and \ref{fixed divisor by two image elements} have many 
computational advantages over Theorem  \ref{t2}, Proposition \ref{t4} and Corollary   
\ref{Bhargava's theorem}. 
Firstly, they do not depend on the evaluation of the factorial of 
$\underline{S}$ or its prime factorization. In fact, due to the sharpness of 
Theorems \ref{Bound for FD for polynomial of type (m,k)}, \ref{Evrard's bound} 
and Corollary \ref{Bhargava's theorem}, it might be possible to use these results 
to evaluate the factorial in some cases. Second, there is no additional step of 
computing the coefficients in alternate bases which essentially amounts to 
inverting matrices with coefficients in $R$. Finally, there is the additional 
freedom in the choice of $\underline{a}$ to minimize the number of primes 
involved in the construction of $\underline{a}_i$.

\end{document}